\documentclass[12pt]{article}
\usepackage{amssymb, amsmath, amsthm}
\usepackage[left=1.25in, right=1.25in, top=1.2in, bottom=1.3in]{geometry}
\usepackage{tikz,xcolor}
\usepackage{varwidth}
\usetikzlibrary{math}

\tikzset{line0/.style={red!90!black, solid}}
\tikzset{line1/.style={blue!100!black, dashed}}
\tikzset{line2/.style={green!60!black, densely dotted}}
\tikzset{line3/.style={black, dashdotted}}
\tikzset{line4/.style={violet, thick, loosely dotted}}

\newtheorem{theorem}{Theorem}[section]
\newtheorem{lemma}[theorem]{Lemma}
\newtheorem{corollary}[theorem]{Corollary}

\newtheorem{proposition}[theorem]{Proposition}

\newtheorem{question}[theorem]{Question}
\newtheorem{observation}[theorem]{Observation}

\def\beq{\begin{equation}}\def\eeq{\end{equation}}
\def\beqn{\begin{eqnarray}}\def\eeqn{\end{eqnarray}}

\title{Monochromatic balanced components, matchings, and paths in multicolored complete bipartite graphs}

\author{Louis DeBiasio\thanks{Department of Mathematics, Miami University, Oxford, Ohio. \texttt{debiasld@miamioh.edu}, \texttt{kruegera@miamioh.edu}} \thanks{Research supported in part by Simons Foundation Collaboration Grant \#283194}
\and
Andr\'as Gy\'arf\'as\thanks{Alfr\'ed R\'enyi Institute of Mathematics, Hungarian Academy of Sciences, Budapest, P.O. Box 127, Budapest, Hungary, H-1364. \texttt{gyarfas.andras@renyi.mta.hu}, \texttt{ruszinko.miklos@renyi.mta.hu}, \texttt{sarkozy.gabor@renyi.mta.hu}} \thanks{Research supported in part by
NKFIH Grant No. K116769.}
\and Robert A. Krueger\footnotemark[1] \and Mikl\'os Ruszink\'o\footnotemark[3]\thanks{Research supported in part by
NKFIH Grant No. K116769.} \and G\'{a}bor N. S\'ark\"ozy\footnotemark[3]
\thanks{Computer Science Department, Worcester Polytechnic Institute, Worcester, MA.} \thanks{Research supported in part by
NKFIH Grants No. K116769, K117879.}
}

\date{}
\begin{document}
 \maketitle

\begin{abstract}
It is well-known that in every $r$-coloring of the edges of the complete bipartite graph $K_{n,n}$ there is a monochromatic connected component with at least $\frac{2n}{r}$ vertices. It would be interesting to know whether we can additionally require that this large component be balanced; that is, is it true that in every $r$-coloring of $K_{n,n}$ there is a monochromatic component that meets both sides in at least $n/r$ vertices?

Over forty years ago, Gy\'arf\'as and Lehel \cite{GYL} and independently Faudree and Schelp \cite{FS2} proved that any $2$-colored $K_{n,n}$ contains a monochromatic $P_n$.  Very recently,
Buci\'c, Letzter and Sudakov \cite{BLS} proved that every $3$-colored $K_{n,n}$ contains a monochromatic connected matching (a matching whose edges are in the same connected component) of size $\lceil n/3 \rceil$.  So the answer is strongly ``yes'' for $1\leq r\leq 3$.

We provide a short proof of (a non-symmetric version of) the original question for $1\leq r\leq 3$; that is, every $r$-coloring of $K_{m,n}$ has a monochromatic component that meets each side in a $1/r$ proportion of its part size.  Then, somewhat surprisingly, we show that the answer to the question  is ``no'' for all $r\ge 4$.  For instance, there are $4$-colorings of $K_{n,n}$ where the largest balanced monochromatic component has $n/5$ vertices in both partite classes (instead of $n/4$).  Our constructions are based on lower bounds for the $r$-color bipartite Ramsey number of $P_4$, denoted $f(r)$, which is the smallest integer $\ell$ such that in every $r$-coloring of the edges of $K_{\ell,\ell}$ there is a monochromatic path on four vertices.  Furthermore, combined with earlier results, we determine $f(r)$ for every value of $r$.
\end{abstract}

\section{Introduction, results}
Let $P_m, C_m$ denote the path and the cycle on $m$ vertices respectively.
There are many results about monochromatic connected components of edge colored graphs and hypergraphs. Here we just refer to surveys \cite{GYSUR}, \cite{GYSUR2} and  \cite{KL}. This note is about the case when the colored host graph is a complete bipartite graph (see \cite[Section 3.1]{GYSUR}).  We start with the following result about the size of the largest monochromatic connected component (for brevity referred here as a monochromatic component).

\begin{theorem}[\cite{GY1}]\label{gy} In every $r$-coloring of the edges of $K_{m,n}$ there is a monochromatic component with at least $\frac{m+n}{r}$ vertices.
\end{theorem}

Mubayi \cite{MU} and Liu, Morris, and Prince \cite{LMP} obtained independently a stronger result: one can require that the monochromatic component in Theorem \ref{gy} is a {\em double star} (a tree obtained by joining the centers of two disjoint stars by an edge).
A slight possible improvement of Theorem \ref{gy} was conjectured in \cite{BGY}: the size of the  monochromatic component could be at least $\lceil {m\over r}\rceil + \lceil {n\over r}\rceil$.

Here we address another natural possible improvement of Theorem \ref{gy}, asking whether the large component can be balanced.

\begin{question}\label{que} Is it true that in every $r$-coloring of the edges of $K_{m,n}$ there is a monochromatic component that intersects the partite classes in at least $m/r$ and $n/r$ vertices, respectively?
\end{question}

For the diagonal case, $m=n$, Question \ref{que} has been studied in stronger forms.  The most important examples are connected matchings (a matching whose edges are in the same connected component), even paths, and cycles. For $r=2$ an affirmative answer to Question \ref{que} has been known in its strongest form for more than forty years: Gy\'arf\'as and Lehel \cite{GYL} and independently Faudree and Schelp \cite{FS2} proved that any $2$-colored $K_{n,n}$ contains a monochromatic $P_n$. For $r=3$ an affirmative answer was recently provided by Buci\'c, Letzter and Sudakov \cite{BLS} (In fact, the authors of this note independently proved the same result, but with a less elegant proof).
\begin{theorem}[\cite{BLS}]\label{main} In every $3$-coloring of the edges of $K_{n,n}$
there is a monochromatic connected matching of size $\lceil n/3 \rceil$.
\end{theorem}

The significance of connected matchings is that with the connected matching-Regularity Lemma method established by {\L}uczak \cite{L}, it is possible to transfer results on connected matchings to asymptotic results for paths and even cycles. For example, this method is used to transfer Theorem \ref{main} to asymptotic results for even cycles and paths in \cite{BLS}.  For similar applications see for example \cite{BS}, \cite{FL}, \cite{GRSS}, \cite{KSS}, \cite{LSS}.

In this note we provide a short proof which answers Question \ref{que} in the affirmative for the non-diagonal case when $1\leq r\leq 3$.

\begin{theorem} \label{boththird} Let $1\le r \le 3$.  In every $r$-coloring of the edges of $K_{m,n}$ there is a monochromatic component that intersects the partite classes in at least $m/r$ and $n/r$ vertices, respectively.
\end{theorem}

However, somewhat surprisingly, we provide a construction which answers Question \ref{que} negatively for all $r\geq 4$.  For instance, there are $4$-colorings of $K_{n,n}$ where the largest balanced monochromatic component has $n/5$ vertices in both partite classes (instead of $n/4$).  Our constructions are based on the $r$-color bipartite Ramsey number of $P_4$, denoted $f(r)$, defined as the smallest integer $\ell$ such that in every $r$-coloring of the edges of $K_{\ell,\ell}$ there is a monochromatic $P_4$.  While for complete host graphs, multicolor Ramsey numbers of $P_4$ have been determined (see \cite[Section 6.4.2]{RAD}) their bipartite analogue has seemingly not been studied explicitly.  However, it turns out that the multicolor bipartite Ramsey number of $P_4$ is equivalent to a well studied graph parameter, the {\em star arboricity} of a graph $G$, denoted $st(G)$, defined to be the minimum number of star forests\footnote{We define a \emph{star} to be a tree having at most one vertex of degree greater than one (this includes isolated vertices and single edges) and a \emph{star forest} to be a forest in which each component is a star.} needed to partition the edge set of $G$.  Then, since a bipartite graph is $P_4$-free if and only if it is a star forest, we get

\begin{observation} $f(r)-1$ is the largest $n$ for which $st(K_{n,n})=r$.
\end{observation}

\begin{theorem}\label{fr} $f(1)=2,f(2)=3,f(3)=4,f(4)=6$ and for $r\ge 5$, $f(r)=2r-3$.
\end{theorem}

Regarding Theorem \ref{fr}, the cases $r=1,2$ are trivial and the case $r=3$ follows from Theorem \ref{main} (and it easy to prove directly).  While it is not hard to see that $f(4)\le 6$, it took some time (and faith) to find a $4$-coloring of $K_{5,5}$ that does not contain a monochromatic $P_4$.  In fact, a computer search later showed that there is (up to isomorphism) only one such coloring.  In the language of star arboricity, Egawa, Fukuda, Nagoya and Urabe \cite{EFNU} gave a proof of Theorem \ref{fr} when $r\ge 5$.  While their inductive step for the upper bound is nice, they do not address the problem how to launch the induction, i.e.\ they do not prove a base case.  In Section \ref{small}, we correct this oversight by proving Theorem \ref{fr}.  Finally, Section \ref{const} provides the lower bound construction for Theorem \ref{fr}.

Blowing up the Ramsey graphs with $f(r)-1$ vertices we get the following.

\begin{proposition}\label{prop:lbound} Let $r,k$ be positive integers and $n=(f(r)-1)k$. There exists an $r$-coloring of $K_{n,n}$ such that every monochromatic component intersects one of the sides in exactly ${n\over f(r)-1}$ vertices.  In particular, the size of the largest monochromatic connected matching is ${n\over f(r)-1}$ and the largest monochromatic monochromatic even path/cycle has ${2n\over f(r)-1}$ vertices.
\end{proposition}

\begin{proof} Let $G$ be a complete balance bipartite graph with $f(r)-1$ vertices in each part, colored with $r$ colors so that there is no monochromatic $P_4$. Replacing each vertex by $k$ vertices and using the color of $uv$ for all edges between the sets replacing $u,v$, we have the required coloring.
\end{proof}

Note that ${n\over f(r)-1}=\frac{n}{5}$ for $r=4$, and ${n\over f(r)-1}=\frac{n}{2r-4}$ for $r\geq 5$.
We can get a lower bound on the size of a monochromatic connected matching in an $r$-colored $K_{n,n}$ for $r\geq 4$ by considering the majority color class which has at least $n^2/r$ edges and applying the following Erd\H{o}s-Gallai-type result for bipartite graphs proved by Gy\'arf\'as, Rousseau, and Schelp \cite{GRS}: If $G$ is a balanced bipartite graph on $2n$ vertices with at least $n^2/r$ edges, then $G$ has a path on at least $(1-\sqrt{1-2/r})n$ vertices (which implies a connected matching of size $\frac{1}{2}(1-\sqrt{1-2/r})n$). Note that for $r\geq 4$, a simple calculation shows 
\[
\frac{1}{2r-1}<\frac{1}{2}(1-\sqrt{1-2/r})< \frac{1}{2r-4}.
\]
Improving the bounds for $r\geq 4$ would be very interesting.

\section{Balanced components}\label{proof}

\begin{proof}[Proof of Theorem \ref{boththird}] Assume that $K=K_{m,n}$ has partite classes $X,Y$ with $|X|=m, |Y|=n$ and the three colors are $A,B,C$. The theorem is trivial for $r=1$. Assume that $r=2$.  By Theorem \ref{gy} there is a monochromatic component, say $A_1$ in color $A$ with at least ${m+n\over 2}$ vertices. Set $X_1=A_1\cap X, Y_1=A_1\cap Y$. We may assume that $|Y_1|<n/2$ and $|X_1|>m/2$, otherwise $A_1$ satisfies the claim of the theorem. Now the biclique $[X_1,Y\setminus Y_1]$ is monochromatic in color $B$ and satisfies the claim of the theorem.

For $r=3$ we proceed similarly. By Theorem \ref{gy} there is a monochromatic component, say $A_1$ in color $A$ with at least ${m+n\over 3}$ vertices. Set $X_1=A_1\cap X, Y_1=A_1\cap Y$. We may assume that $|Y_1|<n/3$ and $|X_1|>m/3$, otherwise $A_1$ satisfies the claim of the theorem. The biclique $K_1=[X_1,Y\setminus Y_1]$ is colored with colors $B,C$. Applying the $r=2$ case to $K_1$, we get a monochromatic component, say $B_1$ in color $B$ such that $|B_1\cap (Y\setminus Y_1)|>{1\over 2}{2n\over 3}=n/3$. The set $X_1\setminus B_1$ is nonempty, otherwise $|X_1\cap B_1|=|X_1|>m/3$ and $B_1$ satisfies the claim of the theorem.

Note that the biclique $[X_1\setminus B_1 ,B_1\cap (Y\setminus Y_1)]$ is monochromatic in color $C$, determining a monochromatic component $C_1$ in $K_1$.  We extend the components $B_1,C_1$ to components $B_1^*,C_1^*$ of $K$, by using all edges of color $B$ and $C$ that go from $B_1\cap (Y\setminus Y_1)$ to $X\setminus X_1$. If $|X\cap B_1^*|$ or $|X\cap C_1^*|$ is at least $m/3$, the component $B_1^*$ or $C_1^*$ satisfies the claim of the theorem. Otherwise $|B_1^*\cup C_1^*|<2m/3$ and the biclique $[X\setminus (B_1^*\cup C_1^*), B_1\cap (Y\setminus Y_1)]$ is monochromatic in color $A$ with at least $n/3,m/3$ vertices in its partite classes, giving the desired monochromatic component.
\end{proof}

Colorings of $K_{m,n}$ where all monochromatic components are complete bipartite graphs are called {\em bi-equivalence colorings}. It was conjectured by Gy\'arf\'as and Lehel \cite{GY1} that bi-eqivalence $r$-colorings of complete bipartite graphs have vertex coverings by at most $2r-2$ monochromatic components. This was proved for $r\le 5$ in \cite{CGyLT}. Applying Theorem \ref{boththird} to bi-eqivalence $3$-colorings, we get the following corollary.

\begin{corollary} \label{largebiclique} In every bi-eqivalence $3$-coloring of $K_{m,n}$ there exists a monochromatic $K_{\lceil m/3\rceil, \lceil n/3\rceil}$.
\end{corollary}

Corollary \ref{largebiclique} is sharp.  Let $m_1\geq m_2\geq m_3$ be integers, as equal as possible, such that $m_1+m_2+m_3=m$, and likewise for $n_1\geq n_2 \geq n_3$.  Consider the unique 1-factorization of $K_{3,3}$ coloring the edges of each matching with a different color, then blow-up the vertices of $K_{3,3}$ into vertex sets of sizes $m_1, m_2, m_3$ and $n_1, n_2, n_3$ respectively, extending the coloring in the natural way. In the resulting coloring, the largest monchromatic complete bipartite graph has $m_1=\lceil m/3\rceil$ vertices on one side and $n_1=\lceil n/3\rceil$ vertices on the other.

\section{The Ramsey number $f(4)$ and $f(r)\leq 2r-3$ for $r\geq 5$}\label{small}

\begin{proposition}\label{f4} $f(4)=6$, i.e. $st(K_{5,5})=4$.
\end{proposition}
\begin{proof}

First note that the upper bound follows from Lemma \ref{56}. The construction showing $f(4)\geq 6$ is defined as follows. Denote the vertex set of $K_{5,5}$ by $\{A_1, \dots, A_5, B_1, \dots, B_5\}$ where the $A_i$'s form one side of the bipartition and the $B_i$'s form the other. All color classes have three components and components but one are $P_3$-s. The exceptional component is a star with three edges. We use the convention that $XYZ$ denotes the path with edges $XY,YZ$ and $X;A,B,C$ denotes the star with edges $XA,XB,XC$.

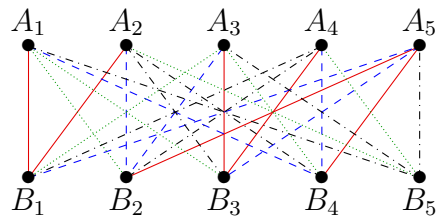
\begin{figure}[h]
\centering
\begin{tikzpicture}

\newcommand\width{1.3};
\newcommand\height{1.75};

\node[left] at (0, 1) 
    {\begin{varwidth}{\linewidth}\begin{enumerate}
\item $A_1B_1A_2, A_3B_3A_4, B_2A_5B_4$
\item $A_1B_4A_4, A_2B_2A_3, B_1A_5B_3$
\item $A_2B_5A_4, B_1A_3B_4, B_2A_1B_3$
\item $B_1A_4B_2, B_3A_2B_4, B_5; A_1A_3A_5$
    \end{enumerate}\end{varwidth}};

\foreach \i in {1,...,5} {
	\coordinate (A\i) at (\i*\width,\height);
	\coordinate (B\i) at (\i*\width,0);
}

\draw[line0] (A1) -- (B1) -- (A2);
\draw[line0] (A3) -- (B3) -- (A4);
\draw[line0] (B2) -- (A5) -- (B4);

\draw[line1] (A1) -- (B4) -- (A4);
\draw[line1] (A2) -- (B2) -- (A3);
\draw[line1] (B1) -- (A5) -- (B3);

\draw[line2] (A2) -- (B5) -- (A4);
\draw[line2] (B1) -- (A3) -- (B4);
\draw[line2] (B2) -- (A1) -- (B3);

\draw[line3] (B1) -- (A4) -- (B2);
\draw[line3] (B3) -- (A2) -- (B4);
\draw[line3] (A1) -- (B5) -- (A3);
\draw[line3] (B5) -- (A5);

\foreach \i in {1,...,5} {
	\draw (A\i) node{\textbullet};
	\draw (A\i) node[above]{$A_\i$};
	\draw (B\i) node{\textbullet};
	\draw (B\i) node[below]{$B_\i$};
}

\end{tikzpicture}
\caption{The construction showing $f(4)\geq 6$.}\label{fig:f4}
\end{figure}
~
\end{proof}

\begin{proposition}\label{rr} For all integers $r$ with $r\geq 5$, $f(r)=2r-3$, i.e. $st(K_{2r-3,2r-3})=r+1$.
\end{proposition}

The proof uses the induction idea of \cite{EFNU} but we (necessarily) reduce the $K_{7,7}$ case to the $K_{5,6}$ case. The {\em center vertex} of a star is its unique vertex of maximum degree, except for the one-edge star, in which case we arbitrarily choose one of the two vertices to be designated as the center.  We say that a star is non-trivial if it has at least one edge.  In a $P_4$-free coloring of a complete bipartite graph a vertex is a {\em special center vertex of color $i$} if it is a center vertex of a star of color $i$ but not a center vertex of any star of any other color. Note that special central vertices $v,w$ of color $i$ cannot be in different partite classes, otherwise the edge $vw$, which has color $j\neq i$, is not incident with the center of a star in color $j$, a contradiction.

We first prove the following lemma.

\begin{lemma}\label{56} In every 4-coloring of $K_{5,6}$ there exists a monochromatic $P_4$.
\end{lemma}
\begin{proof} Let $X,Y$ be the vertex classes of $K_{5,6}$ with $|X|=5$ and $|Y|=6$ and suppose for contradiction that we are given a $P_4$-free $4$-coloring of the edges. Suppose there are two special center vertices $y_1,y_2\in Y$ of the same color, say red. Since $y_1$ and $y_2$ are each incident with five edges, at most one of each color other than red, both $y_1$ and $y_2$ have red degree at least 2. Now any four red neighbors of $y_1,y_2$ together with the four vertices of $Y\setminus \{y_1, y_2\}$ define a $P_4$-free $K_{4,4}$, colored with three colors, contradicting the fact that $f(3)=4$.  Thus $Y$ contains at most one special center vertex of each color.

Now suppose that there are two special center vertices $x_1,x_2\in X$ of the same color, say red.  Since $x_1$ and $x_2$ are each incident with six edges, at most one of each color other than red, both $x_1$ and $x_2$ have red degree at least 3.  So $x_1$ and $x_2$ have red degree exactly 3, implying that the red color class is completely determined, i.e. it has two 3-edge stars with centers $x_1,x_2$.  Since $|X|=5$, there are at most two colors which have exactly two special center vertices in $X$.

Combining the two observations above with the fact that special center vertices of the same color cannot appear on opposite sides of the bipartition, we have that there are at most six special center vertices total.  Since every vertex is a center of at least two star components, except for the special central vertices, there are at least $2\times 11- 6=16$ star components total.  The number of edges in color class $i$ is $11-c_i$ where $c_i$ is the number of components in color class $i$ (including trivial stars).  So the total number of edges is at most $4\times 11-16=28<30$, a contradiction.
\end{proof}

\begin{proof}[Proof of Proposition \ref{rr}.] Let $r$ be an integer with $r\geq 5$ and suppose that every $(r-1)$-coloring of $K_{2r-5, 2r-4}$ has a monochromatic $P_4$ -- note that Lemma \ref{56} provides the base case.  Let $X,Y$ be the vertex classes of $K_{2r-3,2r-3}$ and suppose for contradiction that we are given a $P_4$-free $r$-coloring of the edges.  Suppose that there are two special center vertices of the same color, say red, on the same side of the bipartition, say $x_1, x_2\in X$.  Since $x_1$ and $x_2$ are each incident with $2r-3$ edges, at most one of each color other than red, both $x_1$ and $x_2$ have red degree at least $2r-3-(r-1)=r-2$. Thus we can select a set $A$ in $Y$ consisting of $r-2$ red neighbors of $x_1$ and $r-2$ red neighbors of $x_2$. Then the complete bipartite graph $[A,X\setminus \{x_1, x_2\}]$ is a $P_4$-free $K_{2r-5,2r-4}$ colored with $(r-1)$-colors, contradicting the inductive hypothesis.  Combined with the fact that special center vertices of the same color cannot appear on opposite sides of the bipartition, this implies that there is at most one special center vertex in all colors, a total of at most $r$ special center vertices.

Since every vertex is a center of at least two star components, except for the special center vertices, there are at least $2\times 2(2r-3)-r=7r-12$ star components. The number of edges in color class $i$ is $2(2r-3)-c_i$ where $c_i$ is the number of components in color class $i$ (including trivial stars).  So the total number of edges is at most $$r\times 2(2r-3)-(7r-12)=4r^2-13r+12=(2r-3)^2-(r-3)<(2r-3)^2,$$ a contradiction.
\end{proof}

\section{Construction: $f(r)\ge 2r-3$}\label{const}

Yongqi, Yuansheng, Feng, and Bingxi \cite{YYFB} gave new lower bounds for the multicolor Ramsey numbers of paths and even cycles by constructing a coloring of the {\em complete graph} $K_{2r-4}$ with $r-1$ colors such that all but one of the color classes are the union of two stars of size $r-2$ and one color class forms a perfect matching. Note that this construction also shows that $st(K_{2r-4})\leq r-1$. Indeed, in the language of star arboricity, this example was independently discovered by Akiyama and Kano \cite{AK}.

In \cite{EFNU}, the authors give an example to show that $st(K_{n,n})\leq \lceil n/2\rceil+2$ for $n\geq 7$ which implies $f(r)\geq 2r-3$ for $r\geq 5$.  However, it is worth showing how to transform the $(r-1)$-coloring of $K_{2r-4}$ given above into the $r$-coloring of $K_{2r-4,2r-4}$ required for Theorem \ref{fr} (which was how we discovered the lower bound originally).  The transformation used here is explored many times in graph theory, transforming $K_n$ into $K_{n,n}$ by replacing its vertices by a $1$-factor and its edges by symmetric pairs of edges.

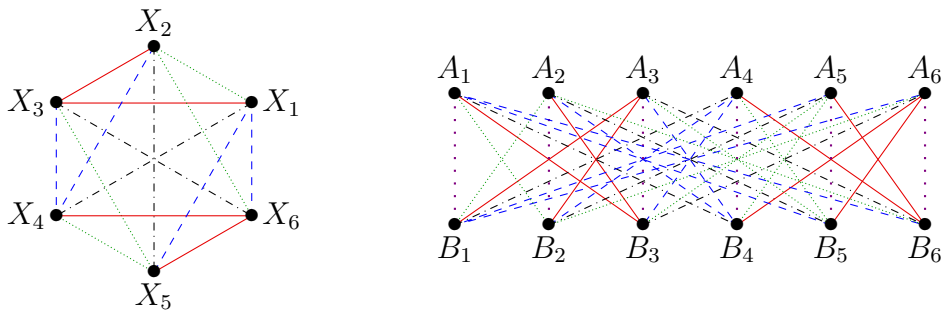
\begin{figure}[ht!]
\centering
\begin{tikzpicture}
\newcommand\radius{1.5};
\newcommand\width{1.25};
\newcommand\height{1.75};

\foreach \i in {0,...,5} {
	\coordinate (\i) at ({\radius*cos(60*\i+30)},{\radius*sin(60*\i+30)});
}

\draw (0) node[right]{$X_1$};
\draw (1) node[above]{$X_2$};
\draw (2) node[left]{$X_3$};
\draw (3) node[left]{$X_4$};
\draw (4) node[below]{$X_5$};
\draw (5) node[right]{$X_6$};

\foreach \i in {0,...,5} {
	\tikzmath{\j=int(Mod(\i+1,6)); \k=int(Mod(\i+2,6)); \c=int(Mod(\i,3)); \o=int(Mod(\i+3,6));}
	\draw[line\c] (\i) -- (\k) -- (\j);
	\ifnum \i<3
	\draw[line3] (\i) -- (\o);
	\fi
}

\foreach \i in {0,...,5} {
	\draw (\i) node{\textbullet};
}

\begin{scope}[shift={({\radius+2.5},-.5*\height)}]

\foreach \i in {0,...,5} {
	\coordinate (A\i) at (\i*\width,\height);
	\coordinate (B\i) at (\i*\width,0);
	\tikzmath{\j=int(\i+1);}
	\draw (A\i) node[above]{$A_\j$};
	\draw (B\i) node[below]{$B_\j$};
}

\foreach \i in {0,...,5} {
	\tikzmath{\j=int(Mod(\i+1,6)); \k=int(Mod(\i+2,6)); \c=int(Mod(\i,3)); \o=int(Mod(\i+3,6));}
	\draw[line\c] (A\i) -- (B\k) -- (A\j);
	\draw[line\c] (B\i) -- (A\k) -- (B\j);
	\draw[line4] (A\i) -- (B\i);
	\ifnum \i<3
	\draw[line3] (A\i) -- (B\o);
	\draw[line3] (B\i) -- (A\o);
	\fi
}

\foreach \i in {0,...,5} {
	\draw (A\i) node{\textbullet};
	\draw (B\i) node{\textbullet};
}
\end{scope}
\end{tikzpicture}
\caption{The construction showing $f(r)\geq 2r-3$ in the case where $r=5$.}\label{fig:stars}
\end{figure}

Consider the complete graph on vertex set $\{X_1, X_2, \dots, X_{2r-4}\}$ with the following $(r-1)$-coloring, where indices are computed $(\mathrm{mod}~ 2r-4)$. For each $i\in [r-2]$ color class $i$ consists of two vertex disjoint
stars, one centered at $X_i$ with leaves $X_{i+1}, \dots, X_{i+r-3}$ and the other centered at $X_{i+r-2}$ with leaves $X_{i+r-1}, \dots, X_{i+2r-5
}$, and color class $r-1$ consists of a matching $\{X_1X_{r-1}, X_2X_r, \dots, X_{r-2}X_{2r-4}\}$ (see Figure \ref{fig:stars}).

Now consider the following $r$-coloring of the complete bipartite graph with vertex sets $\{A_1, \dots, A_{2r-4}\}, \{B_1, \dots, B_{2r-4}\}$:  For all $i\neq j$, color $A_iB_j$ with the color of $X_iX_j$ and for all $i\in [2r-4]$, color $A_iB_i$ with color $r$.  Since each color class is a monochromatic star-forest, there are no monochromatic $P_4$'s.

\end{document}